\documentclass{amsart}
\usepackage{a4wide,amssymb}
\usepackage{tikz-cd}

\def\1{^{-1}}

\newtheorem{theorem}{Theorem}[section]
\newtheorem{lemma}[theorem]{Lemma}

\theoremstyle{definition}

\newtheorem{corollary}[theorem]{Corollary}

\theoremstyle{remark}

\numberwithin{equation}{section}

\begin{document}

\title[]{On vanishing of all fourfold products of the Ray classes in symplectic cobordism.}
\begin{abstract}
This note provides certain computations with transfer associated with projective bundles of Spin vector bundles.
One aspect is to revise the proof of the main result of \cite{B} which says that all fourfold products of the Ray classes are zero in symplectic cobordism.   
\end{abstract}
 
\author{Malkhaz Bakuradze }
\address{Faculty of exact and natural sciences, A. Razmadze Math. Institute, Iv. Javakhishvili Tbilisi State University, Georgia }
\email{malkhaz.bakuradze@tsu.ge } 

\maketitle

\section{Introduction}

The Ray classes \cite{R} $\phi_i\in MSp_{8i-3}$ are indecomposable torsion elements of order two in symplectic bordism ring.  $\phi_i$ arise from the expansion of Conner-Floyd symplectic 
Pontryagin class

\begin{equation*}
pf_1((\eta^1-\mathbb{R})\otimes_{\mathbb{R}}(\zeta-\mathbb{H}))=
s\sum_{i\geq 1}\theta_ipf^i_1(\zeta)
\end{equation*}
in $MSp^4(S^4 \wedge BSp(1))$,  where $s$ is the generator of $MSp^1(S^1)=\mathbb{Z}$, $\eta^1 \to S^1$ is the non-trivial real line bundle and $\zeta \to BSP(1)$ is the canonical $Sp(1)$ bundle. The notation $$\theta_{2i}=\phi_i$$ 
is used in the literature because   $\theta_{2i+1}=0$, for $i>1$ \cite{Ro}. 

\medskip

The classes $\phi_i$ play an essential role in the torsion of the symplectic cobordism ring \cite{R,G,V1}.

By \cite{G, GR} one has $\theta_1\phi_i\phi_j=0$ and $\phi_i^{2i+3}=0$. By \cite{V1} most ternary products $\phi_i\phi_j\phi_k$ are nonzero.

\bigskip

In (\cite{B}, Prop. 4.1) we proved the following 

\begin{theorem}
	\label{main}
	i) All fourfold products of the Ray classes $\phi_i\phi_j\phi_k\phi_l$ are zero;
	
	ii) The images of all double products $\phi_i\phi_j$ in self-conjugate cobordism are zero.
\end{theorem}

\medskip

In this note we revise the proof of Theorem \ref{main} as follows.
 In \cite{B} the arguments of Remark 1.11, Lemma 1.12, and the proof of Proposition 1, (1.1) and (1.2), case $m=5$ are inherited from the references and don't seem to be true. Still, these statements seem to be the consequences of Theorem 3.1 in \cite{BF}. However, all these points are used to derive the proof of Proposition 1 of \cite{B}, which we cover in Section 3. To do this, we first carry out some calculations with transfer in symplectic cobordism by using only double coset formula of \cite{Fe}.  For the reader's convenience, in Section 4 we briefly recall the proof of Theorem \ref{main} by pointing to the sequence of necessary propositions of \cite{B}.

\section{Preliminaries}

Recall from \cite{ABS} the groups $Spin(n)$ and $Pin(n)$  that operate on $\mathbb{R}^n$ by vector representation.

We will use an octonionic representation of Clifford algebra $Cl(8,0)$ and  refer to \cite{SC}. 

One has the isomorphism of Clifford algebras 

\begin{align}
\label{cl}
Cl^0(q + 1,p)\simeq  Cl(p,q)\simeq Cl^0(p,q + 1)
\end{align}
obtained from extending
$$e_1e_{k+1} \leftarrow e_k \rightarrow e_ke_{n+1}, \,\,\,(1 \leq k \leq n).$$ 

This defines the inclusions of $Pin(n)=Pin^0(n)+Pin^1(n)$ in $Spin(n+1)$, where $Pin^0(n)=Spin(n)$.

Let $\{e_0,e_1,\cdots e_7\}$ be an orthonormal basis of $V=\mathbb{R}^8$ . Note that we choose indices  ranging from $0$ to $7$. The octonionic algebra $\mathbb{O}$ is assumed to be given with basis $\{i_0,i_1, \cdots, i_7\}$ obeying the multiplication table  
\begin{align*}
	&i_0 = 1,\,\,\, i_k^2 =-1,\,\,\, i_ki_l = i_m =-i_li_k,\,\,\,1 \leq k \leq 7, \,\,\,\text { and cyclic for }\\
	&  (k,l,m) \in P=\{(1,2,3),(1,4,5),(1,6,7),(2,6,4),(2,5,7),(3,4,7),(3,5,6)\}.
\end{align*}


One can identify $V$ with $\mathbb{O}$ as vector spaces by $\sum x^ke_k\to \sum x^ki_k$.

An octonionic representation $Cl(8,0) \to M_2(\mathbb{O})$ is given by 
\[ \Gamma_k=\gamma_8(e_k)=\left( \begin{array}{cccc}
0 &i_k\\
i_k^* &0 \end{array} \right), \,\,\, 0\leq k \leq 7.\]

\[ \Rightarrow \gamma_8(x)=\left( \begin{array}{cccc}
0 &x\\
x^* &0 \end{array} \right), \,\, x\in V.\]

The carrier space of the representation is understood to be $\mathbb{O}^2$, i.e., the set of columns of two octonions, with $\gamma_8$ acting on it by left multiplication. 

Restricting the representation $Cl(8,0)$ to $Cl^0(8,0)=Cl^0(0,8)$ produces a faithful representation with the generators 

\[ \Gamma_0 \Gamma_k=\gamma_8(e_k)=\left( \begin{array}{cccc}
i_k &0\\
0 &-i_k \end{array} \right) , \,\,\, 1\leq k \leq 7. \]

So $Cl^0(0,8)$ is represented by diagonal matrices. This representation decomposes into two irreducible representations given by the two elements on the diagonal. By the isomorphism $Cl^0(0,8)=Cl(0,7)$ these two are also irreducible representations $Cl(0,7)\to \mathbb{O}$.

Let 
$$\gamma_7:Cl(0,7))\to \mathbb{O}$$
be the irreducible representation given by the generators
\begin{align}
	\label{Cl_7}
	&\gamma_7(e_k)=i_k,\,\,\, 1\leq k\leq 7,\\
	\Leftrightarrow \,\,\,&\gamma_7(x)=Im x, \,\,\, x\in Im \mathbb{O},
\end{align}
which act by successive left multiplication on the carrier space $\mathbb{O}$. 

Orthogonal transformations 
are generated by unit vectors
$u\in Im \mathbb{O}$:
$$x'=\phi_{(\gamma_7 (u)}(x)=u x u^{-1}=-uxu,\,\,\,x\in \mathbb{O}.$$ 

\medskip

By the isomorphism $Cl^0(0,7)\simeq Cl(0,6)$, we obtain a faithful and irreducible representation  $$\gamma_6: Cl(0,6) \to \mathbb{O}:$$
\begin{align}
	\label{Cl_6}
	&\gamma_6(e_k)=i_ki_7,\,\,\, 1\leq k\leq 6,\\
	&\gamma_6(u)=ui_7,\,\, u\in \mathbb{R}^6.
\end{align}

Orthogonal transformations  
are generated by 
\begin{align*}
	x'=&i_k(i_7xi_7)i_k=i_k((i_7xi_7)i_k),\,\,\,x\in \mathbb{O}.
\end{align*}


\bigskip

Using \eqref{cl} for $n=6,5,4,3$ one can see how $Spin(n)$ operates in $R^n$ identified with the imaginary subspace of $\mathbb{O}$ with vanishing $n+1,\cdots, 7$-components.

\section{Spin bundles}

The following bundles induced by the inclusion of groups
\begin{align}
	\label{Sm-1}
	&i_{m}:BSpin(m)\to BSpin(m+1),\\
	\label{Pm-1}
	&j_{m}:BPin(m)\to BSpin(m+1)
\end{align}
can be considered as the sphere bundle and the projective bundle of the universal $Spin(m)$ bundle 
\begin{equation}
\label{xi}
\xi^m \to BSpin(m)
\end{equation}
respectively. 

Denote the sphere bundle and the projective bundle of a vector bundle $\xi$ by $S(\xi)$ and $P(\xi)$ respectively. In particular we have $$S(\xi^m)=BSpin(m),\,\,\,P(\xi^m)=BPin(m)$$ 
and the pullback bundles induced by inclusion  $Spin(m) \hookrightarrow Spin(m+1)$, 
\begin{align}
	&S(\xi^m\oplus 1)\to BSpin(m),\\
	\label{P}
	&P(\xi^m\oplus 1)\to BSpin(m).
\end{align}

\medskip

\begin{lemma}
\label{main lemma}
Let $\xi^{7}\to BSpin(7)$ be the universal $Spin(7)$ bundle as above and
let 
$$\xi=1+\xi^{7}.
$$  
Let $\pi: P(\xi)\to BSpin(7)$ be the projective bundle of $\xi$:
\begin{equation*}
\label{pi}
P(\xi)=ESpin(7)\times_{Spin(7)} RP^7
\end{equation*} 
and let $\mathcal{T}_F(\xi)$ be the tangent bundle along the fibers of $\pi$. Then
$$ \mathcal{T}_{F}(\xi)=\pi^*(\xi^{7}).$$
\end{lemma}

\begin{proof}

Clearly $\phi_{\gamma_7 (u)}$ induces the action of $Spin(7)$ on $\mathbb{O}$, also 
on the real projective space
$$RP^7=\{ \{\pm x\}\big| x\in \mathbb{O}, |x|=1\}$$
and on the tangent bundle of $RP^7$:
$$
\tau_{F}=RP^7\times R^7=\{\pm(x,v(x))\big|v(x)=t_1i_1x+\cdots +t_7i_7x,\,\, t_1, \cdots , t_7 \in \mathbb{R}\}.
$$ 

 $Spin(7)$ acts trivially of on the line in $\mathbb{R}^8=\mathbb{O}$ directed by $i_0$. The action on pure octonions defines the universal $Spin(7)$ bundle $\xi^{7}$.

This defines  
$$
\mathcal{T}_{F}(\xi)=ESpin(7)\times_{Spin(7)}\tau_{F}
$$
and the bundle map
$$
ESpin(7)\times_{Spin(7)}\tau_{F}\to ESpin(7)\times_{Spin(7)} R^7=\xi^{7},
$$
which classifies $\pi^*(\xi^{7})$. 
\end{proof}

\medskip

It is well known that $RP^7$ is paralelizable, i.e., admits $7$ linearly independent tangent vector fields $(\{\pm p, \pm pi_{1}\}),\cdots , (\{\pm p, \pm pi_{7}\})$, where $i_k$ are the octonionic units.  
 
\begin{lemma}
	\label{field}
	  There are $7-k$ number $Spin(k)$-equivariant linearly independent tangent vector fields on $RP^7$, 
	namely $(\{\pm p, \pm pi_{k+1}\}),\cdots , (\{\pm p, \pm pi_{7}\})$, where $k=2,\cdots,6$.
\end{lemma}

\begin{proof}

Let $k=6$ and let us check that the vector field $(\{\pm p, \pm pi_7\})$ on $RP^7$ is invariant under action of $Spin(6)\subset Cl^0_6$:
Using Moufang identities
\begin{align}
	\label{m1}
	&(xyx)z = x(y(xz));\\
	\label{m2}
	&z(xyx) = ((zx)y)x;\\
	\label{m3}
	&x(yz)x = (xy)(zx)
\end{align}
one has for $\{\pm p\}=\pm \{t_0i_0+t_1i_1+\cdots t_7i_7\}$
\begin{align*}
	\phi_{\gamma_6(e_k)}(p)&= i_k((i_7pi_7)i_k)
	=i_k(i_7(t_0i_0+t_1i_1+\cdots t_7i_7)i_7)i_k)\\
	&=i_k(\sum_{n\neq 0,7}(t_ni_n-t_0i_0-t_7i_7)i_k
	=\sum_{n\neq k,7}t_ni_n-t_ki_k-t_7i_7.\\
	\Rightarrow\phi_{\gamma_6(e_je_k)}(p)& =
	\sum_{n\neq j,k}t_ni_n -t_ji_j-t_ki_k;\\
	\Rightarrow\phi_{\gamma_6(e_k)}(i_7)&=-i_7;  \\
	\Rightarrow\phi_{\gamma_6(e_je_k)}(i_7)&=i_7;\\
\end{align*}

\begin{align*}
	\phi_{\gamma_6(e_k)}(pi_7)&=
	(i_ki_7)(pi_7)(i_7i_k) \\
	& =((i_ki_7)p)(i_7(i_7i_k)) && \text{ by \eqref{m3},  } x=i_ki_7,\,\, y=p,\,\,z=i_7                   \\
	&=-((i_ki_7)p)i_k   && \text{ as }  i_7^2=-1                     \\
	&=((i_7i_k)p)i_k    && \text{ as }  i_7i_k=-i_ki_7                       \\
	&=i_7(i_kpi_k)      &&\text{ by \eqref{m2}, }       x=i_k,\,\,p=y,\,\,z=i_7,                    \\
	&=i_7(i_k(t_0i_0+t_1i_1+\cdots t_7i_7)i_k)\\
	&=i_7(\sum_{n\neq 0, k}t_ni_n-t_0i_0-t_ki_k)\\
	&=(-\sum_{n\neq k,7}t_ni_n+t_ki_k+t_7i_7)i_7\\
	\Rightarrow &\phi_{\gamma_6(e_je_k)}(pi_7)=(\sum_{n\neq j, k}t_ni_n-t_ji_j-t_ki_k)i_7.\\
	\Rightarrow &\phi_{\gamma_6(e_je_k)}(pi_7)= \phi_{\gamma_6(e_je_k)}(p)i_7.
\end{align*}

%
%

Similarly for $k=5,4,3$.

\end{proof}

\medskip

\begin{corollary}
	\label{spin6}
	Let  $\xi^{k}$ be the universal $Spin(k)$ bundle, $k=2,\cdots ,6$.
	Then the tangent bundle along the fibers $RP^7$
	of the projective bundle 
	$$\tilde{\pi}:P(8-k+\xi^{k})\to BSpin(k),$$  
	 admits $(7-k)$ linearly independent sections 
	$$\mathcal{T}_{F}(\xi^{k}+8-k) =\tilde{\pi}^*(\xi^{k})+7-k.$$
\end{corollary}
\begin{proof}

Apply Lemma \ref{main lemma}. For the standard inclusion 
$i_k:BSpin(k)\to BSpin(7)$ one has
$$i_k^*(\xi^{7})=\xi^{k}+7-k,$$
therefore
\begin{align*}
&\mathcal{T}_{F}(i_k^*(\xi^{7}+1))=i_k^*(\tilde{\pi}^*(\xi^{7}))\\
\Leftrightarrow \,\,\, &\mathcal{T}_{F}(\xi^{k}+8-k))=\tilde{\pi}^*(\xi^{k})+7-k.
\end{align*}

Alternatively one can apply Lemma \ref{field} to define $(7-k)$-sections of the tangent bundle along the fibers of $\mathcal{T}_F(i_k^*(\xi))$.

\end{proof}

%
%

%
%

\bigskip

  Let $Tr^*(i_{m-1})$ and $Tr^*(j_{m-1})$ be the transfer homomorphism of \eqref{Sm-1} and \eqref{Pm-1} respectively. Then by naturality of the transfer $i^*_{m}Tr^*(j_{m})$ is the transfer homomorphism of \eqref{P}.

 \begin{lemma} 
 	\label{3-components}
 	Let $2\leq m\leq 7$. The transfer homomorphism of \eqref{P} is the sum of three components, 
 	$$i^*_{m}Tr^*(j_{m})=Tr(j_{m-1})^*-Tr(i_{m-1})^*+Id.$$
 	This corresponds to the endpoints and the interior of the orbit type manifold 
 	$$Spin(m)\big|Spin(m+1)\big|Pin(m)$$
 	which is the line segment. 
 	The corresponding isotropy groups are: $Spin(m)$ at one endpoint,  $Pin(m-1)$ at another endpoint, and $Spin(m-1)$ for the points in the interior.  
 \end{lemma}
 
 \begin{proof}
 	
 	Lemma \ref {3-components} coincides with  Lemma 1.9 and Lemma 1.10 of \cite{B} for $m=4$  and $m=3$ respectively. However for all cases it is convenient to use the octonionic representation of Clifford algebras in Section 2. 
 	
 	By naturality of the transfer map $i^*_mTr(j_m)^*$ coincides with transfer homomorphism of \eqref{P}.
 	
 	Let $m=7$.
 	We consider $RP^{7}$ as $S^7_+=S^7 \cap \{x_{0}\geq 0\}$ with identified antipodal points in $S^{6}=S^7 \cap\{x_{0}= 0\}$. Parametrize $S^7_+$ as
 	$$
 	v=\cos t\cdot i_{0}+\sin t \cdot x , \,\,\, x \in S^{6}\subset Im \mathbb{O}, \,\,\,0\leq t \leq \pi/2.
 	$$

 	Then as above $i_{0}$ is invariant under action of $Spin(7)$ and we have
 	$$v'=i_k(i_7(\cos t\cdot i_{0}+\sin t\cdot x )i_7)i_k=\cos t\cdot i_{0}+\sin t \cdot i_k(i_7xi_7)i_k.$$

 	%
 	%
 	%
 	
 	So the orbit space of the  action of $Spin(7)$ on $RP^{7}$ is the line segment $[0,\pi/2]$: we have three types of orbits: the endpoint $t=0$ corresponds to the pole $e_{0}$, with the isotropy group $Spin(7)$.
 	The endpoint $t=\pi/2$ corresponds to the orbit $RP^6=\{ \pm x\}$, its points have the isotropy groups conjugate to  $Pin(6)$, the isotropy group of $\{\pm i_7\}$. Each point $t\in (0,\pi/2)$ corresponds to the orbit   $\cos t\cdot e_{0}+\sin t \cdot x$, the sphere, consisting of points with the isotropy group conjugate to $Spin(6)$. 
 	
 	Now let $m=6$ and consider $RP^{6}$ as $S^6_+=S(Im\mathbb{O}) \cap \{x_{7}\geq 0\}$ with identified antipodal points in $S^{5}=S^6 \cap\{x_{7}= 0\}$. Parametrize $S^6_+$ as
 		$$
 		v=\cos t\cdot i_{7}+\sin t \cdot x , \,\,\, x \in S^{5}, \,\,\,0\leq t \leq \pi/2.
 		$$  
 		As above $i_{7}$ is invariant under action of $\i_ji_k\in Spin(6)$ and we have
 		$$v'=i_ji_k(\cos t\cdot i_{7}+\sin t\cdot x )i_ki_j=\cos t\cdot i_{7}+\sin t \cdot i_j(i_kxi_k)i_j.$$ 
 	The orbit space of the  action of $Spin(6)$ on $RP^{6}$ is the line segment $[0,\pi/2]$ again: we have three types of orbits: the endpoint $t=0$ corresponds to the pole $e_{7}$, with the isotropy group $Spin(6)$.
 	The endpoint $t=\pi/2$ corresponds to the orbit $RP^5=\{ \pm x\}$, its points have the isotropy groups conjugate to $Pin(5)$, the isotropy group of $\{\pm i_6 \}$. Each point $t\in (0,\pi/2)$ corresponds to the orbit   $\cos t\cdot e_{7}+\sin t \cdot x$, the sphere, consisting of points with the isotropy group conjugate to $Spin(5)$. 
 	
 The proof for $m=5,4,3$ is identical and is left to the reader.

 \end{proof}

\medskip

Consider again the bundles \eqref{Sm-1} and \eqref{Pm-1}. Let  $\lambda \to P(\xi^{m-1})$ be the canonical real line bundle. $\lambda$ splits off the bundle $j^*_{m-1}(\xi^m)$ as the canonical direct summand.  Let $f_{m-1}$ be the classifying map of $\lambda$.

\begin{lemma} 
	\label{relevant}
	
	One has for the composition of the transfer map $Tr_m $ followed by the classifying map $f_m$ is zero  in symplectic cobordism 
	
	\begin{align*}
		i)\,\,\, &i_m^*Tr^*(j_{m}) f^*_{m}=Tr^*(j_{m-1})f^*_{m-1}, \,\,\,&&2\leq m\leq 7 ;\\
		ii) \,\,\, &Tr^*(j_{m})f^*_{m}=0,\,\,\,&&2\leq m\leq 6. 
	\end{align*}	
\end{lemma}

 \begin{proof}
 	
 	For i) apply Lemma \ref{3-components}. 
 	The pullback bundle $i_m^*Tr^*(j_{m})$
 	coincides with \eqref{P}.
 	Then $BSpin(m)$ is simply connected for $m\geq 2$: this is the consequence of the exact sequence of homotopy groups of the fibration $Spin(m)\to ESpin(m)\to BSpin(m)$ and the fact that $Spin(m)$ is path connected for $m\geq 2$ \cite{ABS}. So that only the first component of the transfer homomorphism in Lemma \ref{3-components} is relevant.


 	For ii) let $m=6$.  By Corollary \ref{spin6} transfer map $i_6^*Tr_7$ of 
 	$$
 	i_6^*(j_7)=P(\xi^6\oplus \mathbb{R}^{2})
 	$$
 	is trivial as its fiber $RP^7$ admits $Spin(6)$-equivariant vector field $(p,pi_7)$ and therefore the bundle along the fibers of 
 	$$
 	P(\xi^6\oplus \mathbb{R}^{2})\to BSpim(6)
 	$$
 	admits a section. 
 	Now apply again Lemma \ref{3-components} for $m=7$ and then $m=6$ to complete the proof of ii) for $m=6$.
 	
 	Then for $m<6$ the transfer homomorphism of 
 	$$
 	P(\xi^m\oplus \mathbb{R}^{8-m}),\,\,\,2\leq m < 6
 	$$ 
 	is trivial as the pullback of $i_6^*Tr_7$ by  $i_m$. Apply again Lemma \ref{3-components} for $m+1$ and $m$ to complete the proof of ii).

 \end{proof}

 Recall from \cite{GR} that any $Spin(5)$-bundle is $MSp$-orientable. 
 So is $\xi^m\oplus \mathbb{R}^{8-m}$ for $m\leq 5$. Therefore Lemma \ref{main lemma} implies that $\mathcal{T}_F(P(\xi^m\oplus \mathbb{R}^{8-m}))$ is $MSp$-orientable. Recall from \cite{BG} that in our situation the transfer homomorphism is expressed by Boardman's "Umkern" map which is zero because of zero Euler class.

 \section{Proof of Theorem \ref{main}}
 
 Here we follow the notations of \cite{B}. The proof of Theorem \ref{main} is organized as follows.

The tensor square of the canonical $Sp(1)$-bundle $\zeta \to BSp(1)$  has a trivial summand     
 $$
 \zeta\otimes_H \zeta^*=\Lambda+1,
 $$
 where $\Lambda \to BSp(1)$ 
 is the canonical $Spin(3)$-bundle. 
 
 Let $N$ be the normalizer of the torus $S^1=U(1)$ in $S^3=Sp(1)=Spin(3)$. 
 Clearly the bundle 
 $$p: BN\to BSp(1)=BSpin(3)$$
 is the projective bundle of $\Lambda$.
 The quotient map $N/U(1)=Z/2$  induces the map
 $$f:BN\to BZ/2,$$ 
 the classifying map of the canonical real line bundle 
 \begin{equation}
 \label{lambda^2}
 \lambda\to BN, \,\,\,\lambda^{\otimes 2}=1. 
 \end{equation}
 The pullback of $\Lambda$  splits canonically over  $BN$
 \begin{equation}
 \label{BN}
 p^*(\Lambda)=\lambda+\mu.  
 \end{equation}

 Lemma \ref{relevant} case $m=2$ implies 
 \begin{equation}
 \label{Trf} 
 Tr^*f^*=0 
 \end{equation}
 in symplectic cobordism, where $Tr$ is the transfer map of $p$. 
 
 Then it turns out \cite{GR}, \cite{B} p.4394, that  $\Lambda$ is $MSp$-orientable and the Thom class can be chosen in such a way that its restriction to the zero section is equal to 
 
 \begin{equation}
 \label{e(Lambda)}
 \tilde{e}(\Lambda)=\theta_1+\sum_i\phi_i{x^i},\,\,\,x=pf^i_1(\zeta).
 \end{equation}
 
 Now let $\Lambda_i$ be the pullback of $\Lambda$ induced by projection on $i$-th factor $BSp(1)^4\to BSp(1)$ and $\lambda$ be as above. Then (\cite{B}, Lemma 4.5) 
 
 \begin{align}
 	\label{msp}
 	&\lambda\otimes_{\mathbb{R}}\sum_1^4\Lambda_i\to B\mathbb{Z}/2\times BSp(1)^4 \text{  is $MSp$-orientable},\\	
 	\label{SC}
 	&\lambda\otimes_{\mathbb{R}}\sum_1^2\Lambda_i\to B\mathbb{Z}/2\times BSp(1)^2 \text{  is $SC$-orientable}.
 \end{align} 
 
 Because of \eqref{BN} and \eqref{lambda^2} the pullback of \eqref{msp}  over 
 $$(f,p)\times 1: BN \times BSp(1)^3\to B\mathbb{Z}/2\times BSp(1)^4$$
 has a trivial summand and therefore zero $MSp$-orientation  Euler class. 
 
 Similarly for the pullback of \eqref{SC} over 
 $$(f,p)\times 1: BN \times BSp(1)\to B\mathbb{Z}/2\times BSp(1)^2.$$

 Thus (\cite{B}, Lemma 4.6) one has in $MSp^*(BN \times BSp(1)^3)$
 
 \begin{align}
 	\label{1}
 	0=&\prod_{s=1}^4(\theta_i+\sum_{r\geq 1}\phi_r x_s^{2r})+
 	\sum_{m,n,p,q \geq 0}f^*(\gamma_{mnpq})x^m_1 x^n_2 x^p_3x^q_4\\
 	\label{2}
 	=&\sum_{i,j,k,l\geq 1} \phi_i\phi_j\phi_k\phi_lx_1^{2i}x_2^{2j}x_3^{2k}x_4^{2l}+\sum_{m,n,p,q \geq 0}f^*(\gamma_{mnpq})x^m_1 x^n_2 x^p_3x^q_4.
 \end{align}
 Here in \eqref{1} we use the relation $\theta_1\phi_i\phi_j=0$ of \cite{G}.
 
 Similarly one has in $SC^*(BN \times BSp(1))$

 \begin{align}
 	\label{3}
 	0=\sum_{i,j\geq 1} \phi_i\phi_jx_1^{2i}x_2^{2j}+\sum_{m,n \geq 0}f^*(\gamma_{mn})x^m_1 x^n_2.
 \end{align}

 Finally to complete the proof of Theorem \ref{main} i) apply \eqref{2} and \eqref{Trf}. 
 
 Similarly  apply \eqref{3} and \eqref{Trf} to complete the proof of Theorem \ref{main} ii).

\end{document}